\documentclass{amsart}
\usepackage{amsxtra}
\usepackage[all]{xy}
\usepackage{palatino}

\usepackage[bookmarks=false]{hyperref}

\usepackage{amssymb,amsmath,amsthm,epsf,epsfig,dsfont,bbm}
\def \hG{\widehat{\mathbb{G}}_{\mathrm{a}}} 
\usepackage{upref, eucal}

\newcommand{\inj}{\hookrightarrow}

\newcommand{\n}{\mfrak{n}}

\newcommand{\uomega}{\underline{\omega}}
\newcommand{\barS}{\overline{S}}

\newcommand{\ov}[1]{\overline{#1}}
\numberwithin{equation}{section}
\def \mb{\mbox}

\usepackage{amssymb,amsmath,amsthm,epsf,epsfig,dsfont,bbm}

\usepackage{latexsym}

\usepackage{xcolor}
\newcommand{\nc}{\newcommand}

\newcommand{\Cs}{M'_{bal,U_1(p),H'}}
\newcommand{\Jac}{\mb{Jac}(\Cs)}
\newcommand{\ner}{\mb{\tiny Ner}}
\DeclareMathOperator{\Spec}{\mathrm{Spec}}
\DeclareMathOperator{\Spf}{\mathrm{Spf}}
\DeclareMathOperator{\Lie}{\mathrm{Lie}}

\newcommand{\Hom}{\mathrm{Hom}}
\newcommand{\End}{\mathrm{End}}

\newcommand{\Res}{\mb{Res}}

\newcommand{\beqar}{\begin{eqnarray*}}
\newcommand{\eeqar}{\end{eqnarray*}}

\newcommand{\longmap}{{\,\longrightarrow\,}}

\newcommand{\oldmarginpar}[1]{}

\newcommand{\obar}[1]{\overline{#1}}

\def \d{\delta}
\newcommand{\mcal}[1]{\mathcal{#1}}

\begin{document}
\nc{\Mhf}{{\bf M}_{0,H'}'}
\newcommand{\Muf}{{\bf M}'_{bal.U_1(p),H'}}
\newcommand{\Mho}{M_{0,H}}
\newcommand{\Muo}{M_{bal.U_1(p),H}}
\newcommand{\Mh}{M'_{0,H'}}
\newcommand{\Mu}{M'_{bal.U_1(p),H'}}
\newcommand{\Mucan}{{M'}_{bal.U_1(p),H'}^{{can}}}
\newcommand{\oMh}{\overline{M'}_{0,H'}}
\newcommand{\oMu}{\overline{M'}_{bal.U_1(p),H'}}
\newcommand{\pomega}{\omega^+}
\newcommand{\Eg}{{E'_1}\mid_S}
\newcommand{\Ig}{\mb{Ig}}

\newcommand{\Fab}{F_{\tiny \mb{abs}}}
\newcommand{\Frel}{F_{\tiny \mb{rel}}}
\nc{\sym}{\Psi}

\nc{\Gm}{\mathbb{G}_m}

\newcommand{\MH}{{\widehat{M}}'_{0,H'}}
\newcommand{\MU}{{\widehat{M}}'_{bal.U(p),H'}}
\newcommand{\Xfo}{{\widehat{X}}_!}
\newcommand{\Xfor}{{\widehat{X}}_{!!}}
\newcommand{\Xp}{\overline{X}_!}
\newcommand{\Xpp}{\overline{X}_{!!}}
\newcommand{\opi}{\overline{\pi}}
\newcommand{\ohat}[1]{\widehat{#1}}
\newcommand{\fsharp}{{\bold f}^\sharp_{{\tiny \Pi}}}
\newcommand{\mfrak}[1]{\mathfrak{#1}}
\newcommand{\bA}{{\bf A}}
\newcommand{\Mig}{\overline{M'}_{Ig,H'}}
\newcommand{\gp}{\mathfrak{p}}
\newcommand{\gP}{\mathfrak{P}}
\newcommand{\sA}{\mathcal{A}}

\newcommand{\bF}{\mathbb{F}}

\newcommand{\bb}[1]{\mathbb{#1}}
\newcommand{\cX}{\mathcal{X}}
\newcommand\A{\mathbb{A}}
\newcommand\C{\mathbb{C}}
\newcommand\G{\mathbb{G}}
\newcommand\N{\mathbb{N}}
\newcommand\T{\mathbb{T}}
\newcommand\sE{{\mathcal{E}}}
\newcommand\tE{{\mathbb{E}}}
\newcommand\sF{{\mathcal{F}}}
\newcommand\sG{{\mathcal{G}}}
\newcommand\GL{{\mathrm{GL}}}
\newcommand\bM{\mathbb{M}}
\newcommand\HH{{\mathrm H}}
\newcommand\mM{{\mathrm M}}
\newcommand\J{\mathfrak{J}}
\newcommand\sT{\mathcal{T}}
\newcommand\Qbar{{\bar{\Q}}}
\newcommand\sQ{{\mathcal{Q}}}
\newcommand\sP{{\mathbb{P}}}
\newcommand{\Q}{\mathbb{Q}}
\newcommand{\tH}{\mathbb{H}}
\newcommand{\Z}{\mathbb{Z}}
\newcommand{\R}{\mathbb{R}}
\newcommand{\F}{\mathbb{F}}
\newcommand\fP{\mathfrak{P}}
\newcommand\Gal{{\mathrm {Gal}}}
\newcommand{\Ou}{\mathcal{O}}
\newcommand{\legendre}[2] {\left(\frac{#1}{#2}\right)}
\newcommand\iso{{\> \simeq \>}}
\newcommand{\M}{\mathbb{M}}
\newcommand{\m}{\mathcal{m}}

\newtheorem {theorem} {\bf{Theorem}}[section]
\newtheorem {lemma}[theorem] {\bf Lemma}
\newtheorem {prop}[theorem]{\bf Proposition}
\newtheorem {proposition}[theorem] {\bf Proposition}
\newtheorem {exercise}{Exercise}[section]
\newtheorem {example}[theorem]{Example}
\newtheorem {remark}[theorem]{Remark}
\newtheorem {definition}[theorem]{\bf Definition}
\newtheorem {corollary}[theorem] {\bf Corollary}

\newtheorem{thm}{Theorem}
\newtheorem{cor}[thm]{Corollary}
\newtheorem{conj}[thm]{Conjecture}
\newtheorem{quen}[thm]{Question}
\theoremstyle{definition}
\newtheorem{claim}[thm]{Claim}
\newtheorem{que}[thm]{Question}
\newtheorem{lem}[thm]{Lemma}
\def \vp{\varphi}
\newcommand{\wh}[1]{\widehat{#1}}
\newcommand{\pip}{\Pi}
\newcommand{\Xdag}{\widehat{X}_\dagger}

\theoremstyle{definition}
\newtheorem{dfn}{Definition}

\theoremstyle{remark}


\newcommand {\stk} {\stackrel}
\newcommand{\map}{\rightarrow}
\def \bX{{\bf X}}
\newcommand {\orho}{\overline{\rho}}

\newcommand{\hGm}{\wh{\G}_m}

\newcommand{\Ner}[1]{{#1}^{\ner}}
\newcommand{\red}[1]{{\color{red} #1}}
\newcommand{\blue}[1]{{\color{blue} #1}}

\theoremstyle{remark}
\newtheorem*{fact}{Fact}
\makeatletter
\def\imod#1{\allowbreak\mkern10mu({\operator@font mod}\,\,#1)}
\makeatother

\title{Differential  modular forms over totally real fields of 
integral weights}
\author{Debargha Banerjee}
\address{Indian Institute of Science Education and Research Pune, 
Dr. Homi Bhabha Road, Pashan, Pune 411 008, India}
\email{debargha@iiserpune.ac.in}
\author{Arnab Saha}
\address{Indian Institute of Technology Gandhinagar, Palaj, 
Gandhinagar  382355, India}
\email{arnab.saha@iitgn.ac.in}
\begin{abstract}
In this article, we construct a differential modular form of non-zero order
and integral weight for compact Shimura curves
over totally real fields bigger than $\Q$.  The construction uses the theory of 
mod $p$ companion forms by Gee and the lift of Igusa curve to characteristic 
0. This is the analogue of the  
construction of Buium in \cite{MR3349440}.  
\end{abstract}
\maketitle

\section{Introduction}

The theory of $\d$-geometry, due to A. Buium and A. Joyal, has developed as an 
arithmetic analogue of 
differential algebra. In this theory, the role of a derivation is played by
a $p$-derivation $\d$. Similar to the way that the algebraic definition of a 
usual 
derivation comes from the power series ring, a $p$-derivation is defined using 
the $p$-typical Witt vectors. A $p$-derivation $\d$ on any ring $A$ satisfies
\beqar
\d(x + y) & = & \d x + \d y + \frac{x^p + y^p -(x+y)^p}{p}, \\
\d (xy) & = & x^p \d y + y^p \d x + p\d x \d y. \\
\eeqar
for all $x,y \in A$.
Such a ring $A$ with a $p$-derivation $\d$ on it is called a $\d$-ring. 
The theory of arithmetic jet spaces on algebraic groups (e.g. $\GL_n$) 
was developed in the following series of papers \cite{MR3473428,MR3477327,
MR3477328}. In \cite{buium2019perfectoid}, a canonical perfectoid space
is attached to jet spaces by using convergence properties of 
$\d$-characters.

In a recent development by Bhatt and Scholze on comparison theorems, the 
prismatic sites defined in \cite{bhatt2019prisms} are $\d$-rings satisfying 
certain 
divisorial conditions. Here they show that the various cohomology
theories such as the de Rham, crystalline and \'{e}tale can be obtained by
'base changing' the prismatic cohomology.
In \cite{drin,isocrys} the $\d$-geometry leads to remarkable
 new weakly admissible filtered isocrystals which do not come from crystalline 
cohomology. Then the Fontaine functor associates  new 
$p$-adic Galois representations to such objects.

In \cite{MR1908022, MR1748272, MR2025806, MR2081150, 
MR2882615,  MR2890522, MR2890519} 
the arithmetic jet space theory was developed over a modular
curve and its associated Hodge bundle. Sections of the arithmetic jet space 
associated to the Hodge bundle are called differential modular forms. 
For every $n$, there exist canonical morphisms $\phi$ 
from the $n$-th jet space to 
the $(n-1)$-th jet space which are lifts of the Frobenius map.
Hence one considers the bundle obtained
from the pull-back of 
the Hodge bundle along various compositional powers of such $\phi$ and taking
their tensor products.  
A section of such a bundle is called a {\it differential modular form of order 
$n$ and weight $w$}, where $w$ belongs to the weight space which is the 
ring of polynomials $\Z[\phi]$ and $n$
is the degree of $w$. We will denote the space of such differential
modular forms as $M^n(w)$. The precise definition is given in Section 
\ref{dmf}.

One of the striking features of this theory comes from the existence of 
differential modular forms which do not have any classical counterpart. 
We note down a few and their properties:
\begin{enumerate}
\item {\bf $f^1$}: \cite{MR1748272} This is a differential modular form in 
$M^1(-\phi-1)$ that admits $\d$-Fourier expansion of the form
$$
\Psi= \frac{1}{p}\log \frac{\phi(q)}{q^p} =  \sum_{n \geq 1} (-1)^{n-1}n^{-1}
p^{n-1} \left(\frac{q'}{q^p}\right)^n 
$$
where $q$ is the usual Fourier parameter and $q'$ is the formal $\d$-coordinate
associated to $q$. The interesting feature of $f^1$ is that when it is 
evaluated on the $R$-points of the modular curve, then its zero locus 
precisely consists of the elliptic curves $E$ that are canonical lifts, that is 
$E$ has a lift of Frobenius on its structure sheaf; in otherwords the 
Serre-Tate parameter of $E$ is $1$.

\item {\bf $f^\partial$}: \cite{MR1988499} This is a differential modular form 
in 
$M^1(\phi-1)$ whose $\delta$-Fourier expansion is $1$ and is a characteristic
$0$ lift of the Hasse invariant.

\item {\bf $f^\sharp$}: \cite{MR2400054} 
Given a Hecke newform $f(q)=\sum_{n\geq 1} a_n q^n$ 
of weight 2, $f^\sharp$ is obtained from  
the $\d$-character of the modular elliptic curve associated to $f$. Its
$\d$-Fourier expansion is given by
$$
\frac{1}{p}\sum_{n\geq 1} \frac{a_n}{n} (\phi^2(q)^n - a_p \phi(q)^n + pq^n).
$$ 
\end{enumerate}
The forms $f^1$ and $f^\sharp$ and their diophantine properties 
led to the result of finite intersection of Heegner points
and finite rank subgroups on a modular correspondence \cite{MR2507742}.

In \cite{MR3128464} the first author extended the theory of differential
modular forms to the setting of totally real fields setting and 
developed the objects analogous to $f^1, ~f^\partial$ and $f^\sharp$. 
In \cite{MR2890522} Buium constructed a new differential modular form coming 
from 
mod $p$ newforms. This paper is devoted to the construction of
the analogous object
associated to a mod $p$ Hilbert modular form $\Pi$ of level $\n$ prime to 
$p$ and weight $k$.

We now explain our result in greater detail. 
Let $F$ be a totally real field of degree $d > 1$ over $\Q$
with $\tau_1,\cdots, \tau_d$ the infinite places of $F$. Let $\gp_1,\cdots,
\gp_m$ be the primes of $F$ lying above $p$. Fix $\gp = \gp_1$ and let 
the residue field of $F$ at $\gp$ be of cardinality $q$ which is a 
power of $p$. Let $\Ou_\gp$ be the completion of the local ring at $\gp$,
$F_\gp$ be its fraction field and $q$ be the cardinality of the residue
field $\Ou_{\gp}/\gp$.
Let $R$ to be the completion of the maximal
unramified extension of $\Ou_{\gp}$ with $(\pi)$ the maximal ideal. 
Let $\kappa= R/(\pi)$ be the residue field which is algebraically closed 
and $K$ be the fraction field
of $R$. In our context, we will be considering the general notion of 
$\pi$-derivation $\d$ pertaining to the uniformiser $\pi$ of the maximal ideal
of $R$.

Consider $R$ along with its unique lift of Frobenius $\phi$ on it. The 
associated $\pi$-derivation $\d$ on $R$ is given by
$$
\d r = \frac{\phi(r) - r^q}{\pi}
$$
for $r \in R$. Given  a $\pi$-formal scheme $Z$ over $\Spf R$, the $n$-th jet 
space $J^nZ$ as functor of points is defined as
$$
J^nZ(B) := Z(W_n(B))
$$
for all $\pi$-adically complete $R$-algebras $B$ and $W_n(B)$ is the ring of 
$\pi$-typical
Witt vectors. Then $J^nZ$ is representable by a $\pi$-formal scheme and the 
restriction and the Frobenius maps from $W_n(B)$ to $W_{n-1}(B)$ induce
scheme-theoretic morphisms, denoted $u$ and $\phi$ respectively, from
$J^nZ$ to $J^{n-1}Z$. This makes the system of $\pi$-formal schemes $\{J^nZ
\}_{n=0}^\infty$ a canonical object in the category of prolongation sequences 
of $\pi$-formal schemes. Clearly, if $Z$ is a $\pi$-formal group scheme then 
$J^nZ$ also is a group object.

Given an ordinary
 mod $p$ Hilbert modular form $\Pi$ of level $\n$ prime to $p$ and weight $k$, 
by \cite{MR2357747} Gee associates an ordinary companion form $\Pi'$ of 
parallel weight 
$k'= p+1 -k$ and level $\n$. To such a $\Pi'$, by Hida theory one associates 
an ordinary mod 
$p$ form of weight $2$ and level $\n p$. By abuse of notation, we will still
denote this ordinary weight 2 mod $p$ form by $\Pi'$. Now by Serre lifting one 
can associate a 1-form $\omega_{\Pi'} \in H^0(\Mu,\Omega_{\Mu})$ which is
then a weight 2 Hecke cuspform and is the characteristic 0 lifting of the 
ordinary mod $p$ form $\Pi'$. Here the unitary Shimura curve $\Mu$ is a finite
and flat cover of $\Mh$ and we refer to Section \ref{quaternionas} for 
the recollection of basic definitions and properties of these curves. 
We consider the following map of schemes over $\Spec K$ 
\begin{equation}
\label{abeljacobi}
\Jac \stk{\nu}{\longrightarrow} A_{\Pi'}
\end{equation}
where $A_{\Pi'}$ is the quotient abelian scheme over $\Spec K$
associated to $\omega_{\Pi'}$. 

For any scheme $Z$ over $\Spec K$, let $\Ner{Z}$ denote its N\'{e}ron model
over $\Spec R$. Then in Section \ref{neron}, we associate the following 
morphism of group schemes over $\Spec R$
\begin{equation}
\label{abeljacobineron}
(\Ner{\Jac})^0 \map B,
\end{equation}
where $B$ is either an abelian scheme or a split torus of dimension $g$ where
$g$ is the dimension of $\Jac$ over $\Spec K$ and $(\Ner{\Jac})^0$ is the
connected component of the identity of $\Ner{\Jac}$.

Given any scheme $Y$ over $\Spec R$, let 
$\wh{Y}$ denote its $\pi$-formal completion. 
Then we consider an affine open $p$-formal subscheme $\wh{X}$ of 
$\wh{\Ner{\Mh}}$ which we assume is also contained inside the ordinary locus.
We still denote by $L$ the restriction and the $\pi$-formal completion of 
the
Hodge bundle (minus the zero section) on $\wh{X}$.  Then the fibration $L \map 
X$ is a $\hGm$-bundle and induces the fibration $J^nL \map J^nX$ which is a
$\bb{W}_n^*$-bundle where $\bb{W}_n^*$ is the $\pi$-formal group scheme of the
multiplicative units of the Witt vector scheme $\bb{W}_n$.

The differential modular forms of order $n$ are sections of $J^nL$. The 
weight space of such differential modular forms are precisely the 
multiplicative characters of the $\pi$-formal scheme $\bb{W}_n^*$ which is 
$\Z[\phi]$.
For $2 < k < p$ we define $k'$ as its conjugate if $0 < k' <q-1$ and there 
exists an integer $c$ such that $k' \equiv c(k-2) \bmod (q-1)$. Then our main 
result is:

\begin{theorem}
\label{main-thm-new}
For any mod $p$ Hilbert modular form $\pip$ of level $\mfrak{n}$ and weight 
$2 < k < p $,
there exists a differential modular form $\fsharp$ of order either $1$ or $2$
and integral weight $k'$ where $k'$ is a conjugate of $k$.
\end{theorem}

{\bf Strategy of Proof.} 
In Section \ref{main-result} we construct $\Xdag$ such that 
the reduction mod $p$ of $\Xdag$ is contained inside the Igusa curve
which is an \'{e}tale cover of $\ov{X}$ of degree $q-1$. We
consider the following composition of $\pi$-formal schemes
\begin{equation}
\Xdag \inj \MU \map \wh{B}.
\end{equation}
Therefore for all $n$, the associated morphism of jet spaces will be
$J^n\Xdag \map J^n \wh{B}.$
Hence composing the above morphism with any non-zero order $n$ differential 
character $\Theta_n:J^n\wh{B} \map \hG$ (which exists for $n=1$ when $\wh{B}$
is a split torus and for $n=2$ when $\wh{B}$ is a $\pi$-formal abelian
scheme over $\Spf R$ \cite{MR1358979}),
we obtain a non-zero section $\fsharp$ on $\Ou(J^n\Xdag)$.

Now $J^n\Xdag \map J^n\wh{X}$ is \'{e}tale since $\Xdag \map \wh{X}$ is and 
therefore $\Ou(J^n\Xdag)$ is a finite graded module over $\Ou(J^nX)$ where 
the gradation respects the Galois group of $\Xdag$ over $X$ (which is 
$\Z/(q-1)\Z$). Each graded piece is the space of differential modular forms
of an appropriate weight. Hence it is enough to show that $\fsharp$ belongs
to a graded piece and that follows from certain compatibility results of the
maps with the Galois group.

 \section{Acknowledgement}
The first author was partially supported by the SERB grant YSS/2015/001491 and
MTR/2017/000357.
The second author is also grateful to the Max Planck Institute for Mathematics 
in Bonn for its hospitality and financial support.
The authors wish to thank Alexandru Buium and James Borger for several 
inspiring discussions and clarifications. 
We would also like to thank Alexei Pantchichkine and 
Jack Shotton for the discussions and 
insights that took place during the preparation of this paper.

\section{Notations}
\begin{enumerate}
\item[$\bullet$] Let $F$ be a totally real field of degree $d > 1$ over $\Q$
with $\tau_1,\cdots, \tau_d$ the infinite places of $F$. Let $\gp_1,\cdots,
\gp_m$ be the primes of $F$ lying above $p$. Fix $\gp = \gp_1$ and let 
$\kappa$ be the residue field of $F$ at $\gp$ with cardinality $q$ which is a 
power of $p$. Let $\Ou_\gp$ be the completion of the local ring at $\gp$ and
$F_\gp$ be its fraction field and let $q$ be the cardinality of the residue
field $\Ou_{\gp} /\gp$.

\item[$\bullet$] Let $R$ to be the completion of the maximal
unramified extension of $\Ou_{\gp}$ with $(\pi)$ the maximal ideal 
and let $\kappa= R/(\pi)$ be the residue field which is algebraically closed
and let $K$ be the fraction field of $R$.
 
\item[$\bullet$] Let $B$ be a quaternion algebra over $F$ that splits exactly
at one infinite place, say $\tau_1$, such that 
\subitem{$\bullet$} $B$ splits at $\gp$
\subitem{$\bullet$} Fix a maximal order $\Ou_B$ of $B$ and choose an 
isomorphism $\Ou_{B,\nu} \simeq M_2(\Ou_\nu)$ for all finite places $\nu$
of $F$ where $B$ splits
\subitem{$\bullet$} Fix an isomorphism $B_{\tau_1} \simeq M_2(\bb{R})$.

\item[$\bullet$] Let $\lambda <0$ and $K=\Q(\sqrt{\lambda})$ the imaginary
quadratic extension over $\Q$ such that $p$ splits. Consider 
$E=F(\sqrt{\lambda})$
which is an extension of degree $2d$ over $\Q$. 

\item[$\bullet$] Let $D = B \otimes_F E$ and $\Ou_D$ be a maximal order of $D$.
Then we have the following decomposition 
$$\Ou_D \otimes \Z_p = (\Ou_{D_1^1} \oplus \cdots \oplus \Ou_{D_m^1}) 
\oplus (\Ou_{D^2_1} \oplus \cdots \oplus \Ou_{D^2_m}).$$
Then for any $\Ou_D \otimes \Z_p$-module  $\Lambda$ admits a decomposition as
$$\Lambda = (\Lambda_1^1 \oplus \cdots \oplus \Lambda_m^1) \oplus 
(\Lambda_1^2 \oplus \cdots \oplus \Lambda_m^2). $$
The $\Ou_{D^2}$-module $\Lambda_1^2$ decomposes as the direct sum of two
$\Ou_\gp$-modules $\Lambda_1^{2,1}$ and $\Lambda_1^{2,2}$, the kernels of the
respective idempotents $\left(\begin{array}{ll} 1 & 0 \\ 0 & 0 \end{array}
\right)$ and $\left(\begin{array}{ll} 0 & 0 \\ 1 & 0 \end{array}\right)$.

\item[$\bullet$] For any $R$-algebra $A$, let $\overline{A}= A/pA$.

\item[$\bullet$] Let $K$ be the fraction field of $R$.

\item[$\bullet$] 
If  $X$ is a scheme over $\Spec R$, let $X_K:= X \times_{\Spec R} \Spec K$ be 
the generic fiber.  

\item[$\bullet$] $\obar{X}= X \times_{\Spec R} \Spec \kappa$ be the special 
fiber over $\Spec \kappa$.

\item[$\bullet$] 
Let $\ohat{X}$ denote the $\pi$-formal completion of $X$ over $\Spf R$.

\item[$\bullet$]
For any scheme $Z$ over $\Spec K$, let $Z^{\ner}$
denote its N\'{e}ron model over $\Spec R$. 

\end{enumerate}

\section{Shimura curves}
\label{quaternionas}

We first start by recalling the basic notions of the various types of Shimura 
curves. The main references are \cite{MR860139}, \cite{MR2357747} and 
\cite{MR2027194}.

\subsection{Quaternionic Shimura curves}
Let $G= \Res_{F/\Q}(B^\times)$ be the reductive group over $\Q$. Let $K \subset
G(\bb{A}_{\Q}^f)$ be a compact subgroup where $\bb{A}^f_{\Q}$ are 
the finite adeles. Then the Shimura curve associated to $K$ is defined to be
the following:
\begin{align}
\label{Shim1}
M_K(\C) = G(\Q) \backslash (G(\A^f_\Q) \times (\C \backslash \R))/K.
\end{align}
Write $K= K_\gp K^\gp$ where $K_\gp$ be the component corresponding to 
$\gp$ and $K^\gp$ for all the rest of the finite places. Let $K^\gp= H$ be a
fixed group.
In this article, we will be interested in the following two choices of $K_\gp$: 

(1) $K_\gp = GL_2(\Ou_\gp)$ and we will denote the corresponding 
Shimura curve as $\Mho$.

(2) $K_\gp= \left\{\left(\begin{array}{ll} a & b \\ c & d
 \end{array}\right)  \in GL_2(\Ou_\gp) \mid 
\left(\begin{array}{ll} a & b \\ c & d \end{array}\right) \equiv
\left(\begin{array}{ll} 1 & * \\ 0 & 1 \end{array}\right) 
\bmod \gp \right\} $
and we will denote the Shimura curve as $\Muo$.

\subsection{Unitary Shimura Curves $M_{K'}'$}
Now we will give a brief introduction to the unitary Shimura curves. The 
following theorem of Carayol in \cite{MR860139} connects the unitary Shimura 
curves $M_{K'}$ with the quaternionic ones denoted $M_K$, once they are 
both base changed to $\Spec R$:

\begin{thm}[Carayol]
\label{Carayol}
Let $H \subset \Gamma$ be a small enough open compact subgroup and $N_H$ a 
connected
component of $M_{0,H} \times_{\Spec \Ou_\gp} \Spec R$. There exists an open 
compact subgroup 
$H' \subset {\Gamma}'$
and a connected component $N'_{H'}$ of $M'_{0,H'} \times_{\Spec \Ou_\gp} 
\Spec R$, such that $N_H$ and $N'_{H'}$ are isomorphic over $\Spec R$.
\end{thm}

Fix $\mu$ to be a square root of $\lambda$ and consider the map $E \map 
F_\gp \oplus F_\gp$ given by $(x+y\sqrt \lambda) \mapsto (x+y\mu, x-y\mu)$, 
which extend to an isomorphism 
\begin{align}
\label{ringiso}
E \otimes \Q_p \simeq F_p \oplus F_p \simeq (F_{\gp_1} \oplus \cdots \oplus
F_{\gp_m}) \oplus (F_{\gp_1} \oplus \cdots \oplus F_{\gp_m}).
\end{align}
The above gives an inclusion of $E$ in $F_\gp$ via the projection
$$E \inj E\otimes \Q_p \simeq F_p \oplus F_p \stk{\tiny \mb{pr}_1}{\map} F_p 
\stk{\tiny \mb{pr}_1}{\map} F_\gp.$$
Let $z \mapsto \overline{z}$ denote the conjugation of $E$ with respect to $F$.
Let $D=B \otimes_F E$ and let $l \mapsto \overline{l}$ be the canonical 
involution of $B$ with the conjugation of $E$ over $F$. Let $V$ be the 
underlying $\Q$-vector space of $D$. Choose $\Delta \in D$ such that 
$\ov{\Delta} = \Delta$ and define an involution on $D$ by $l^*:= 
\Delta^{-1} \ov{l} \Delta$. Choose $\alpha \in E$ such that $\ov{\alpha} = 
-\alpha$. Define a symplectic form $\sym$ on $V$ as 
$$\sym(u,w) = \mb{tr}_{E/\Q}(\alpha \mb{tr}_{D/E}(v\Delta w^*)).$$ 
The symplectic form $\sym$ is an alternating non-degenerate form on $V$ and
satisfies 
$$\sym(lv,w) = \sym(v,l^*w).$$
Let $G'$ be the reductive algebraic group over $\Q$ such that for any 
$\Q$-algebra $B$, $G'(B)$ is the group of $D$-linear symplectic simplitudes
of $(V\otimes_\Q B, \sym\otimes_\Q B)$.

Let $\Ou_B$ be a fixed maximal order of $B$ and fix an isomorphism $\Ou_B
\otimes_{\Ou_F} \Ou_\gp \simeq M_2(\Ou_\gp)$. Let $\Ou_D$ be a maximal order
of $D$. Let $V_\Z$ denote the corresponding lattice in $V$. The 
decomposition of $E\otimes \Q_p$ induces the following decomposition of 
$D \otimes \Q_p$ and $\Ou_D\otimes \Z_p$
$$\xymatrix{
\Ou_D \otimes \Z_p \ar@{^{(}->}[d] \ar@{=}[r]&  (\Ou_{D^1_1} \oplus \cdots \oplus \Ou_{D^1_m}) 
\oplus (\Ou_{D_1^2} \oplus \cdots \oplus \Ou_{D^2_m}) \ar@{^{(}->}[d] \\
D \otimes \Q_p \ar@{=}[r]& 
(D_1^1 \oplus \cdots \oplus D^1_m) \oplus (D_1^2 \oplus
\cdots \oplus D_m^2) 
}$$
where each $D_j^k$ is an $F_{\gp_j}$-algebra isomorphic to $B \otimes_F
F_{\gp_j}$. One can choose $(\Ou_D,\alpha,\Delta)$ in such a way that

i) $\Ou_D$ is stable under involution $l \mapsto l^*$

ii) each $\Ou_{D_j^k}$ is a maximal order in $D^k_j$ and $\Ou_{D^2_1} \inj
D_1^2= M_2(F_\gp)$ identifies with $M_2(\Ou_\gp)$

iii) $\sym$ takes integer values on $V_\Z$

iv) $\sym$ induces a perfect pairing $\sym_p$ on $V_{\Z_p} = V_\Z \otimes 
\Z_p$

Then each $\Ou_D \otimes \Z_p$-module $\Lambda$ admits a decomposition
\begin{align}
\label{decomp}
\Lambda = (\Lambda_1^1 \oplus \cdots \oplus \Lambda_m^1) \oplus (\Lambda_1^2
\oplus \cdots \oplus \Lambda_m^2)
\end{align}
such that $\Lambda_j^k$ is an $\Ou_{D_j^k}$-module. Also further, the 
$M_2(\Ou_\gp)$-module  $\Lambda_1^2 = \Lambda_1^{2,1} \oplus \Lambda_1^{2,2}$
where the $\Ou_\gp$-modules $\Lambda_1^{2,1}$ and $\Lambda_1^{2,2}$ are 
projections with respect to idempotents $e$ and $1-e$ respectively where 
$e= \left(\begin{array}{ll} 1 & 0 \\ 0 & 0 \end{array} \right)$.
Then the finite adelic points of $G'$ can be described as 
\begin{align}
\label{Gprime}
G'(\A^f)= \Q_p^* \times GL_2(F_\gp) \times \Gamma' 
\end{align}
where 
\begin{align}
\label{Gammaprime}
\Gamma'= G'(\A^{f,p}) \times (B \otimes_F F_{\gp_2})^* \times \cdots \times
(B \otimes_F F_{\gp_m})^*.
\end{align}
Let $K' \subset G'(\A^f)$ be an open compact subgroup and $X'$ be a 
conjugacy class in $G'(\R)$ as in \cite{MR2027194}, page 362. Then the unitary 
Shimura curve over $\C$ is 
\begin{align}
\label{UShimura}
M_{K'}'(\C) = G'(\Q) \backslash G'(\A^f) \times X' / K' 
\end{align}
which is a compact Riemann surface. Let $\ohat{T}(A) = \Pi_p(T_p(A))$ denote 
$\varprojlim_n A[n]$ as a sheaf over $\Spec B$ in the \'{e}tale 
topology. We will consider subgroups of the form 
$$K'= \Z_p^* \times GL_2(\Ou_\gp)\times H' \inj \Q_p^* \times GL_2(F_\gp)
\times \Gamma'.$$
Let $\ohat{T}^p(A) = \Pi_{l \ne p} T_l(A)$. 
As in (\ref{decomp}) the $\Ou_D \times \Z_p$-module $T_p(A)$ decomposes as
$$T_p(A) = \left( (T_p(A))^1_1 \oplus \cdots \oplus (T_p(A))^1_m\right)
\oplus \left( (T_p(A))^2_1 \oplus \cdots \oplus (T_p(A))^2_m \right).
$$
Let us define the following:
\begin{enumerate} 
\item[$\bullet$] $T_p^\gp := (T_p(A))^2_2 \oplus \cdots \oplus (T_p(A))^2_m.$

\item[$\bullet$] $\ohat{W}^p:= V_\Z \otimes \ohat{Z}^p$.

\item[$\bullet$] $W_p^\gp = (V_{\Z_p})^2_2 \oplus \cdots \oplus 
(V_{\Z_p})^2_m$.
\end{enumerate}

Now consider the following functor
$$\Mhf: \{\Ou_\gp\mb{-algebras}\} \map \mb{Sets}$$
where for any $\Ou_\gp$-algebra $B$,  $\Mhf(B)$ is the set of isomorphism 
classes of tuples $(A,i, \theta, \ov{\alpha}^\gp)$ such that 

\begin{enumerate}
\item $A$ is an abelian scheme over $B$ of relative dimension $4d$, equipped 
with an action $\Ou_D$ given by $i: \Ou_D \map \End_B(A)$ such that

\begin{enumerate}
\item the projective $B$-module $\Lie_2^{2,1}(A)$ has rank one and $\Ou_\gp$ 
acts on it via $\Ou_\gp \map B$.

\item for $j \geq 2$, $\Lie_j^2=0$.

\end{enumerate}

\item $\theta$ is a polarisation of $A$ of degree prime to $p$ such that the 
corresponding Rosati involution sends $i(l)$ to $i(l^*)$.

\item $\ov{\alpha}^\gp = \alpha_p^\gp \oplus \alpha^p: T_p^\gp(A) \oplus
\ohat{T}(A) \simeq W_p^\gp \oplus \ohat{W}^p$ modulo $H'$, is a class of 
isomorphism with $\alpha_p^\gp$ linear and $\alpha^p$ symplectic.
\end{enumerate}

The above moduli problem is fine and is represented by a scheme $\Mh$ over
$\Spec \Ou_\gp$. There is an universal object 
$(A'_{0,H'},i,  \theta , \ov{\alpha}_{\gp})$ over $\Mh$
such that any test object over an $\Ou_\gp$-algebra $B$ is obtained by pulling 
back the universal quadruple via
the corresponding morphism $\Spec B \rightarrow \Mh$. Let 
$\sigma: A'_{0,H'} \rightarrow \Mh$
denote the morphism of the universal family to $\Mh$. The $\Ou_{\Mh}$-module 
$\sigma_*(\Omega^1_{A'_{0,H'} / \Mh})$ is
an $O_D \otimes \Z_p$ module and define
\begin{align}
\label{uomega}
\underline{\omega}=
\left(\sigma_*(\Omega^1_{A'_{0,H'} / \Mh})\right)_1^{2,1}
\end{align}
which is a an invertible sheaf 
on $\Mh$. 

Recall $\Eg= (A_\gp)^{(2,1)}_1$ from \cite{MR2357747}, page 6.
We define a $bal.U_1(\gp)$-structure on an $\Mh$-scheme $S$ as a short 
exact sequence of f.p.p.f $\Ou_\gp$-group schemes on $S$ 
$$0 \map \mcal{K} \map \Eg \map \mcal{K}' \map 0$$
such that $\mcal{K}, \mcal{K}'$ are both locally free of rank $q$ together with
sections $P \in \mcal{K}(S)$ and $P' \in \mcal{K}'(S)$ such that they generate
the respective group schemes. Now define the functor
\begin{align}
\Muf: & \{{\bf Schemes}/\Mh\} \longrightarrow {\bf Sets} \\
& S \mapsto \{bal.U_1(\gp)\mb{-structures on } S \}. \nonumber
\end{align}
Then by lemma 2.7 in \cite{MR2357747} the functor $\Muf$ is representable by a 
scheme $\Mu$ over $\Mh$.
We will denote the natural projection map as $\epsilon:\Mu 
\rightarrow M'_{0,H'}$.
The reduction modulo $p$ of $\Mu$, denoted by 
$\overline{M'}_{bal \cdot U_1(p),H'}$, has two irreducible components which 
intersect
each other at the supersingluar points. One of the components is the Igusa
curve, denoted by $\Mig$. By abuse of notation, we will still denote
the induced map on the closed fibers as $\epsilon: \Mig \rightarrow
\overline{M'}_{0,H'}$. We recall lemma $2.8$ in \cite{MR2357747}.

\begin{lemma}
The scheme $M'_{bal, U_1(p), H'}$ is regular of dimension two and we have a 
finite and flat map $\epsilon:M'_{bal, U_1(p), H'} \rightarrow M'_{0,H'}$.
\end{lemma}

\subsection{The section $a^{+}$} 

We will now recall some basic facts from \cite{MR2357747}. 
Let $S$ be an $\oMh$-scheme. For any scheme $Z$ over $\Spec k$ consider 
the diagram
$$\xymatrix{
Z \ar@{.>}[rd]|--{\Frel} \ar@/^/[rrd]^\Fab \ar@/_/[rdd] & &\\
& Z^{(q)} \ar[d] \ar[r] & Z \ar[d] \\
& S \ar[r]_{\Fab} & S 
}$$
where $\Fab$ is the absolute Frobenius over $\Spec k$, $Z^{(q)} = 
Z \times_{S, \Fab} S$  and $\Frel$ is the induced 
relative Frobenius as defined in the diagram above.

Recall the definition of $\Eg$ which is a subgroup scheme of 
$\mfrak{p}$-torsion points of the abelian scheme $A$ over $S$, \cite{MR2357747}
page 6.
Then for any $S$, consider the morphism of group schemes
$\Frel: \Eg \map \Eg^{(q)}$.
We define the Verschiebung $V: \Eg^{(q)} \map \Eg$ which is obtained by 
applying the Cartier duality to the morphism $\Frel$ above.
Now consider the following functor:
$$\Ig:  \{\oMh\mb{-schemes} \} \map \underline{\mb{Sets}}
$$
given by $\Ig(S)= \{P \in \Eg^{(q)}~ \mid~ $P$ \mb{ generates the kernel of } 
V\}$.
We recall the following from \cite{MR2357747}, lemma 2.16:

\begin{lemma}
\label{Igusalem}
The functor $\Ig$ is representable by a regular $1$-dimensional scheme 
$\Mig$ over $\Spec k$. It also admits a natural morphism 
$\epsilon: \Mig \map \oMh$ 
which is finite flat of degree $(q-1)$. 

Moreover $\epsilon$ is \'{e}tale over the ordinary locus and is totally 
ramified over the supersingular locus.
\end{lemma}

We define the sheaf $\pomega := \epsilon^* \uomega$ on $\Mig$
Recall that under Cartier duality, $\Eg$ is dual to itself. Hence given a 
$P \in \ker V$ gives us a morphism $g_P: \Eg \map \Gm$.
Hence the invariant differential 
$dx/x$ on $\Gm$ pulls back to the invariant differential $g_P^*(dx/x)$ on
$\Eg$ and upon restriction, we get an invariant differential on
$\ker(\Frel|_A)$. Then there exists a unique invariant $1$-form on the abelian
scheme $A$ over $S$, that is an element in $H^0(A,\Omega^1_{A/S})$ whose
restriction to $\ker(\Frel|_A)$ is $\phi_P^*(dx/x)$. Hence taking 
$S = \Mig$, we define 
the section $a^+ \in H^0(S,\pomega)$ as described above. 

From \cite{MR3128464} or \cite{MR2027194}, 
recall the definition of the Hasse invariant $H$, which is a mod $p$ modular 
form of weight $(q-1)$. Then following a 
similar argument as in \cite{MR1074305}, proposition $5.2~ (2)$ we obtain

\begin{lemma}
\label{root}
The
section $a^+ \in \omega^+$ is a $(q-1)$-th root of Hasse invariant. 
In other words, 
\[
(a^+)^{q-1}=H,
\]
where $H$ is the Hasse invariant.
\end{lemma}

Let us denote by $(ss)$ the set of supersingular points of $\oMh$ 
and  $\Sigma = \epsilon^{-1}(ss)$.
Let $\wh{X} \subset \MH \backslash (ss)$ be a $\pi$-formal
 affine subscheme and let 
$\overline{X}$ be the reduction mod $\pi$ of $X$. Let 
$\wh{Z}:= \epsilon^{-1}(\wh{X}) \subset \MU$. 
Then $\wh{Z}$ has two connected components since the closed fiber
of $M'_{bal \cdot U(p),H'}$ over $p$ has so.
Let $\wh{X}_!$ denote the component whose reduction mod $p$ is contained inside the Igusa curve $M'_{Ig,H'} \backslash \Sigma$.

\section{Witt Vectors and Arithmetic Jet Spaces}
Witt vectors over Dedekind domains with finite residue fields were introduced 
in \cite{MR2833791}. We will give a brief over view in this section.

\subsection{Frobenius lifts and $\pi$-derivations}
Let $B$ be an $R$-algebra, and let $C$ be a $B$-algebra with structure map $u:B \map C$.
In this paper, a ring homomorphism $\psi:B\map C$ will be called a
{\it lift of Frobenius} (relative to $u$) if it satisfies the following:
\begin{enumerate}
\item The reduction mod $\pi$ of $\psi$ is the $q$-power Frobenius relative 
to $u$, that is,
$\psi(x) \equiv u(x)^q \bmod \pi C$.

\item The restriction of $\psi$ to $R$ coincides with the fixed $\phi$ on $R$,
that is, the following diagram commutes
        $$\xymatrix{
        B \ar[r]^\psi & C \\
        R \ar[r]_{\phi} \ar[u] & R. \ar[u] 
        }$$
\end{enumerate}

A \emph{$\pi$-derivation} $\d$ from $B$ to $C$ means a set-theoretic map
$\d:B \map C$ satisfying the following for all $x,y \in B$
        \begin{eqnarray*}
        \label{der}
        \d(x+y) &=& \d (x) + \d (y) + C_\pi(u(x),u(y)) \\
        \d(xy) &=& u(x)^p \d (y) +  \d (x) u(y)^p + \pi \d (x) \d (y),
        \end{eqnarray*}
where $C_\pi(X,Y)$ denotes the polynomial
        $$
        C_\pi(X,Y) = \frac{X^q + Y^q - (X+Y)^q}{\pi} \in R[X,Y],
        $$
such that for all $r\in R$, we have
$$
\d(r) = \frac{\phi(r)-r^q}{\pi}.
$$
When $C=B$ and $u$ is the identity map, we will call this simply a 
$\pi$-derivation on $B$.

It follows that the map $\phi: B \map C$ defined as
        $$
        \phi(x) := u(x)^p + \pi \d (x)
        $$
is a lift of Frobenius in the sense above. Conversely,
for any flat $R$-algebra $B$ with a lift of Frobenius $\phi$, one can define
the $\pi$-derivation $\d(x)= \frac{\phi(x)-x^q}{\pi}$ for all $x \in B$.

\subsection{Witt vectors}
We will define Witt vectors in terms of the Witt polynomials.
For each $n \geq 0$, let us define $B^{\phi^n}$ to be the $R$-algebra
with structure map $R \stk{\phi^n} {\map} R \stk{u}{\map} B$ and define the \emph{ghost rings}
to be the product $R$-algebras
$\Pi^n_{\phi} B = B \times B^{\phi} \times \cdots  \times B^{\phi^n}$ and
$\Pi_{\phi}^\infty B= B \times B^{\phi} \times \cdots$.
Then for all $n \geq 1$ there exists a \emph{restriction}, or \emph{truncation},
map $T_w:\Pi_{\phi}^nB \map \Pi_{\phi}^{n-1}B$ given by $T_w(w_0,\cdots,w_n)= (w_0,\cdots,w_{n-1})$.
We also have the left shift \emph{Frobenius} operators $F_w:\Pi_{\phi}^n B \map \Pi_{\phi}^{n-1} B$ given by
$F_w(w_0,\dots,w_n) = (w_1,\dots,w_n)$. Note that $T_w$ is an $R$-algebra morphism, but
$F_w$ lies over the Frobenius endomorphism $\phi$ of $R$.

Now as sets define
        \begin{equation}
        \label{eq-witt-coord}
        W_n(B)=B^{n+1}, 
        \end{equation}
and define the set map $w:W_n(B) \map \Pi_{\phi}^n B$  by $w(x_0,\dots,x_n)= (w_0,\dots,w_n)$ where
        \begin{equation}
        \label{eq-witt-poly}
        w_i = x_0^{q^i}+ \pi x_1^{q^{i-1}}+ \cdots + \pi^i x_i
        \end{equation}
are the \emph{Witt polynomials}.
The map $w$ is known as the {\it ghost} map. (Do note that under the traditional indexing our $W_n$ would be
denoted $W_{n+1}$.) We can then define the ring $W_n(B)$, the ring
of truncated $\pi$-typical Witt vectors, by the following theorem as for 
example in \cite{MR3316757}, page 141:
\begin{theorem}
\label{wittdef}
For each $n \geq 0$, there exists a unique functorial $R$-algebra structure on $W_n(B)$ such that
$w$ becomes a natural transformation of functors of $R$-algebras.
\end{theorem}

\subsection{Operations on Witt vectors}
\label{subsec-witt-operations}
Now we recall some important operators on the Witt vectors.
They are the unique functorial operators corresponding under the ghost map
to the operators $T_w$, $V_w$, and $F_w$ on the ghost rings defined above.
First, the \emph{restriction}, or \emph{truncation}, maps $T:W_n(B) \map W_{n-1}(B)$
are given by $T(x_0,\dots,x_n) = 
(x_0,\dots, x_{n-1})$. 
There is also the {\it Frobenius} ring homomorphism
$F:W_n(B) \map W_{n-1}(B)$, which can be described in terms of the ghost map.
It is the unique map which is functorial in $B$ and makes the
following diagram commutative
        \begin{equation}
        \xymatrix{
        W_n(B) \ar[r]^w \ar[d]_F & \Pi^n_{\phi} B \ar[d]^{F_w} \\
        W_{n-1}(B) \ar[r]_-w & \Pi_{\phi}^{n-1} B^n
        } \label{F}
        \end{equation}
As with the ghost components, $T$ is an $R$-algebra map but $F$ lies over the Frobenius endomorphism $\phi$
of $R$.

Finally, we have the multiplicative Teichm\"uller map $\theta:B \map W_n(B)$ 
given by $x\mapsto [x]= (x,0,0,\dots)$.

\subsection{Prolongation sequences and jet spaces}
Let $X$ and $Y$ be $\pi$-formal schemes over $S=\Spf R$. We say a pair 
$(u,\d)$ is a {\it prolongation}, and write 
$Y \stk{(u,\d)}{\map} X$, if $u: Y \map X$ is a map of $\pi$-formal schemes 
over $S$ and $\d: \Ou_X \map u_*\Ou_Y$ is a 
$\pi$-derivation making the following diagram commute: 

$$
	\xymatrix{
	R \ar[r] &  u_* \Ou_Y \\
	R \ar[u]^\d \ar[r] &  \Ou_X \ar[u]_\d \\
	} 
	$$ 
Following \cite{MR1748272} (page 103), a {\it prolongation sequence} is a 
sequence of prolongations
	$$
	\xymatrix{
	S & T^0 \ar_-{(u,\d)}[l] & T^1 \ar_-{(u,\d)}[l] & \cdots\ar_-{(u,\d)}[l]},
	$$
where each $T^n$ is a $\pi$-formal scheme over $S$ satisfying 
$$u^* \circ \d = \d \circ u^*
$$
and $u^*$ is the pull-back morphism on the sheaves induced by $u$.
We will often use the 
notation $T^*$ or $\{T_n\}_{n \geq 0}$.
Note that if the  $T^n$ are flat over $S$ then having a 
$\pi$-derivation $\d$ is equivalent to having lifts of Frobenius $\phi:T^{n+1}
\to T^n$.

Prolongation sequences form a category $\mcal{C}_{S^*}$, where a morphism $f:T^*\to U^*$ is 
a family of morphisms $f^n:T^n\to U^n$ commuting with both the $u$ and $\d$, in the evident sense.
This category has a final object $S^*$ given by $S^n=\Spf R$ for all $n$, where each $u$ is the identity and
each $\d$ is the given $\pi$-derivation on $R$.

For any $\pi$-formal scheme $Y$ over $S$ and for all $n \geq 0$ we define the 
$n$-th jet space $J^nX$ (relative to $S$) as 
	$$
	J^nX (Y) := \Hom_S(W_n^*(Y),X)
	$$
where $W_n^*(Y)$ is defined as in 10.3 of \cite{MR2854117}. We will not define 
$W_n^*(Y)$ in full generality here. 
Instead, for simplicity of the exposition,
we will define $\Hom_S(W_n^*(Y),X)$ in the affine case.
Write $X = \Spf A$ and $Y=\Spf B$. Then $W_n^*(Y)= \Spf W_n(B)$ and $\Hom_S(W_{n}^*Y,X)$ is 
$\Hom_R(A,W_n(B))$, the set of $R$-algebra homomorphisms $A\to W_n(B)$.

Then $J^*X:= \{J^nX \}_{n \geq 0}$ forms a prolongation sequence and is 
called the {\it canonical prolongation sequence} as in 
\cite[Proposition 1,1]{MR1748272}. By the same Proposition 1.1 in 
\cite{MR1748272}, $J^*X$ satisfies the following 
universal property---for any $T^* \in \mcal{C}_{S^*}$ and $X$ a $\pi$-formal
scheme over 
$S^0$, we have
\begin{equation}	
\label{canprouniv}
	\Hom(T^0,X) = \Hom_{\mcal{C}_{S^*}}(T^*, J^*X).
\end{equation}	
Let $X$ be a $\pi$-formal scheme over
$S= \Spf R$. Define $X^{\phi^n}$ by $X^{\phi^n}(B) := X(B^{\phi^n})$ for any $R$-algebra $B$.
In other words, $X^{\phi^n}$ is $X \times_{S,\phi^n} S$, the pull-back of $X$ under the map $\phi^n:S\to S$.
(N.B., $(\Spf A)^{\phi^n}$ should not be confused with $\Spf (A^{\phi^n})$.)
Next define the following product of $\pi$-formal schemes
        $$
        \Pi^n_{\phi} X= X \times_S X^{\phi} \times_S \cdots \times_S X^{\phi^n}.
        $$
Then for any $R$-algebra $B$ we have $X(\Pi_{\phi}^n B) = X(B)\times_S \cdots \times_S X^{\phi^n}(B)$.
Thus the ghost map $w$ in Theorem \ref{wittdef} defines a map of $\pi$-formal
$S$-schemes
        $$
        w:J^nX \map \Pi_{\phi}^nX.
        $$
Note that $w$ is injective when evaluated on points with coordinates in any flat $R$-algebra.

The operators $F$ and $F_w$ in (\ref{F}) induce maps $\phi$ and $\phi_w$ fitting into a commutative diagram
\begin{equation}
        \label{phiw}
        \xymatrix{
        J^nX \ar[r]^w\ar[d]_{\phi} & \Pi_{\phi}^n X \ar[d]^{\phi_w}\\
        J^{n-1}X \ar[r]_w & \Pi_{\phi}^{n-1} X.
        }
\end{equation}
The map $\phi_w$ is easier to define. It is the left-shift operator given by
        $$
        \phi_w(w_0,\dots,w_n)= (\phi_S(w_1),\dots,\phi_S(w_n)),
        $$
where  $\phi_S:X^{\phi^i} \map X^{\phi^{i-1}}$ is the composition given in the following diagram:
\begin{equation}
        \label{ko}
        \xymatrix{
        X^{\phi^i}\ar[r]^-{\sim} & X^{\phi^{i-1}} \times_{S,\phi} S \ar[d] 
        \ar[r]^-{} & 
        X^{\phi^{i-1}} \ar[d] \\
        & S \ar[r]_{\phi} & S.
        }
\end{equation}
We note that a choice of a coordinate system on $X$ over $S$ induces coordinate systems on
$X^{\phi^i}$ for each $i$, and with respect to these coordinate systems,
$\phi_S$ is expressed as the identity. One might say that $\phi_S$
applies $\phi$ to the horizontal coordinates and does nothing to the vertical coordinates.

For the map $\phi:J^nX\to J^{n-1}X$, we can define it in terms of the functor of points.
For any $R$-algebra $B$, the ring map $F:W_n(B)\to W_{n-1}(B)$ is not $R$-linear
but lies over $\phi:R\to R$. As $B$ varies, the resulting linearized $R$-algebra maps
  $$
        W_n(B)\to W_{n-1}(B)^{\phi} = W_{n-1}(B^{\phi}),
        $$
induce functorial maps
\begin{equation}
        J^nX(B) = X(W_n(B)) \longmap X(W_{n-1}(B^{\phi})) = J^{n-1}X(B^{\phi}),
\end{equation}
which is the same as giving a morphism $\phi:J^n X\to J^{n-1}X$ lying over $\phi:S\to S$.

If $A$ is a $\pi$-formal group scheme over $S$,
the ghost map $w:J^nA\to \Pi_{\phi}^n A$ and the truncation map
$u:J^nA\to J^{n-1}A$
are $\pi$-formal group scheme homomorphisms over $S$.
On the other hand, the Frobenius maps $\phi:J^nA\to J^{n-1}A$ and $\phi_w:\Pi_{\phi}^nA \to \Pi_{\phi}^{n-1}A$
are $\pi$-formal group scheme homomorphisms lying over
the Frobenius endomorphism $\phi$ of $S$.

\subsection{Character groups of group schemes}
Given a prolongation sequence $T^*$ we can define its shift $T^{*+n}$ by 
$(T^{*+n})^j:= T^{n+j}$ for all $j$, page 106 in \cite{MR1748272}.
	$$
	S \stk{(u,\d)}{\leftarrow} T^n \stk{(u,\d)}{\leftarrow} T^{n+1}\dots 
	$$
We define a {\it $\d$-morphism of order $n$} from $X$ to $Y$ to be a 
morphism $J^{*+n}X \map J^*Y$ of prolongation sequences.
Let $A$ denote a $\pi$-formal group scheme over $S$.
We define a {\it character of order $n$}, $\Theta:A \map \hG$ 
to be a $\d$-morphism of order $n$ from $A$ to $\hG$
which is also a morphism of $\pi$-formal group schemes.
By the universal property of jet schemes as in (\ref{canprouniv}),
 an order $n$ character is equivalent to a homomorphism
$\Theta:J^nA \map \hG$ of $\pi$-formal group schemes over $S$. 
We denote the group of 
characters of order $n$ by $\bX_n(A)$. So we have
$$
	\bX_n(A)=\Hom(J^nA,\hG),
	$$
which one could take as an alternative definition. Note that $\bX_n(A)$ comes with an
$R$-module structure since $\hG$ is an $R$-module $\pi$-formal scheme over $S$. Also the inverse system 
$J^{n+1}A \stk{u}{\map} J^nA$ defines a directed system
$$
	\bX_n(A) \stk{u^*}{\map} \bX_{n+1}(A) \stk{u^*}{\map}\cdots
	$$
via pull back. Each morphism $u^*$ is injective and 
we then define $\bX_\infty(A)$ to be the direct limit $\varinjlim \bX_n(A)$
of $R$-modules.

\subsection{Differential Modular Forms}
\label{dmf}

Let $\wh{X}$ be an affine subscheme of $\MH$ such that the reduction mod $\pi$ 
denoted $\ov{X}$ is contained inside the ordinary locus. 
Let $L=\Spec (\bigoplus_{n \in \Z} \uomega^{\otimes n})$ 
be the physical line bundle attached to the line bundle $\underline{\omega}$ 
with the zero section removed over $X$.
The space of modular forms $M$ on $X$ are the global sections of 
$L$ on $X$~\cite{MR2027194}. Then the space of differential modular forms
of order $\leq n$ are the global sections of $J^nV$.

Recall from proposition (2.2) in \cite{MR2400054} that given any polynomial 
$w=w_0 + w_1 \phi + \cdots + w_n \phi^n$ in $\Z[\phi]$ of 
degree $n$, there exists a differential character $\chi_w: J^n\Gm \map \hG$
satisfying 
$$
\chi_w(\lambda) = \lambda^{w_0} \phi(\lambda)^{w_1} \cdots \phi^n(\lambda)^{w_n}
$$ 
for all invertible $\lambda$. 

Let $T^*=\{T^n\}$ be a prolongation sequence with $T^n = \Spf B_n$ where 
$B_n$ are $R$-algebras. Let $(A,i, \theta, \ov{\alpha}^\gp)$ be a tuple 
over $\Spf B_0$. Set $\uomega_{A/B_0} = 
\left(\sigma_* \Omega^1_{A/T^0}\right)$ where $\sigma$ is the identity
section of $A$ over $T^0$.
Then a differential modular form $f$ of order $\leq n$
and weight $w \in \Z[\phi]$ is a rule which assigns to any 
$(A,i, \theta, \ov{\alpha}^\gp,\omega, T^*)$ where $\omega$ is a 
basis for $\uomega_{A/T^0}$ an element 
$$f(A,i, \theta, \ov{\alpha}^\gp,\omega, T^*) \in B_n$$ 
such that 

$i)~ f(A,i, \theta, \ov{\alpha}^\gp,\omega, T^*)$ only depends on the 
isomorphism class of $(A,i, \theta, \ov{\alpha}^\gp,\omega, T^*)$

$ii)~$ the formation of $f(A,i, \theta, \ov{\alpha}^\gp,\omega, T^*)$ 
commutes with arbitrary base change

$iii)$ for any $\lambda \in B_0^\times$ we have $\chi_w(\lambda)$
$$
f(A,i, \theta, \ov{\alpha}^\gp,\lambda\omega, T^*) = \chi_w(\lambda)
f(A,i, \theta, \ov{\alpha}^\gp,\omega, T^*)
$$
Let $M^n(w)$ denote the subspace of differential
modular forms of order $\leq n$ and of weight $w \in \Z[\phi]$. 

Another way to describe $M^n(w)$ is as follows: consider the invertible sheaf
$$\uomega^{\otimes w} := \uomega^{w_0} \otimes (\phi^*\uomega)^{w_1} \otimes
\cdots \otimes ({\phi^n}^*\uomega)^{w_n}
$$
on $J^nX$. Then $M^n(w)$ is the space of global sections of $\uomega^{\otimes w}
$.
If $x$ globally generates $\underline{\omega}$, then 
every element of $M^n(w)$ is of the form $a x^w$ for some $a \in \Ou(J^n(X))$
and where $x^w= x^{w_0} \cdots (\phi^{n *}x)^{w_n}$. 

\section{Schemes attached to companion modular forms in characteristic zero}

\subsection{Preliminaries}

Let us recall the main Theorem of \cite{MR2357747}:

\begin{theorem}
\label{companion}
Let $\pip$ be a mod $p$ Hilbert modular form of parallel weight $2<k<p$ and
level $\mathfrak{n}$, $\mathfrak{n}$ coprime to $p$. Suppose $\pip$ is 
ordinary at all primes $\mfrak{p} \mid p$ and that the mod $p$ representation
$\orho_\pip:\Gal(\overline{F}/F) \map \mb{GL}_{2}({\ov{\F}}_p)$ is irreducible
and is tamely ramified at all primes $\mfrak{p} \mid p$. Then there is a 
companion form $\pip'$ of parallel weight $k'=p+1-k$ and level $\mfrak{n}$
satisfying $\ov{\rho}_{\pip'} \simeq \ov{\rho}_\pip \otimes \chi^{k'-1}$,
where $\chi$ is the $p$-adic cyclotomic character.
\end{theorem}

Also recall lemma $3.2$ of \cite{MR2357747}:

\begin{lemma}
\label{levelup}
An ordinary mod $p$ Hilbert modular form of level $\mfrak{n}$ prime to $p$ is
an ordinary mod $p$ form of weight $2$ and level $\mfrak{n}p$.
\end{lemma}

Hence starting with a $\pip$ which is a mod $p$ Hilbert modular form of 
parallel weight $k$, a companion form $\pip'$ constructed in theorem 
\ref{companion} is as a weight $2$ form of level $\mfrak{n}p$ 
by lemma \ref{levelup}. 
Consider the identification of the first \'{e}tale
cohomology tensored with $\bold{C}$ with the Hodge decomposition 
$$H^0(\Mu  \otimes \bold{C},\Omega^1_{\Mu}) \oplus
H^0(\Mu \otimes \bold{C},\ov{\Omega}^1_{\Mu}).
$$
Then by the discussion in \cite{MR2357747} page 17 (also definition $4.16$), 
the companion form $\pip'$ 
corresponds uniquely to a differential
$\omega_{\pip'} \in H^0(\Mu  \otimes \bold{C},\Omega^1_{\Mu} )$ where
$\omega_{\pip'}$ is a Hecke eigenform of weight $2$ (we are using the 
identification of the module of differentials 
$H^0(\Mu  \otimes \bold{C},\Omega^1_{\Mu})$
with weight $2$ modular forms).

Now we will describe the action of the diamond operator, that is the 
action of $(\Z/ q\Z)^\times$ on $\omega_{\pip'}$.
As in definition $4.16$ in \cite{MR2357747}, for $\alpha \in (\Z/q\Z)^\times$ 
we have $\langle \alpha \rangle \omega_{\pip'} = 
\theta(\alpha)^{-k'}
\omega_{\pip'}$ where $\theta: (\Z/q \Z)^\times \map 
R$ is the Teichm{\"u}ller character and $k'$ is as in theorem
\ref{companion}. Let $\Jac \stk{\nu}{\map} A_{\pip'}$ 
denote the quotient of abelian schemes (as in Theorem 4.4 in \cite{MR623136}) 
associated to the Hecke eigenform $\omega_{\pip'}$ that satisfies  
\begin{align}
\label{alpha}
\xymatrix{
\Jac \ar[d]_\nu \ar[r]^{\langle \alpha \rangle} & \Jac \ar[d]^\nu \\
A_{\pip'} \ar[r]^{\iota(\alpha)} & A_{\pip'}
}
\end{align}
where $\iota(\alpha) \in \End(A_{\pip'})$ is the induced endomorphism on 
$A_{\pip'}$.
Then the action of $\iota(\alpha)$ is given by
\begin{align}
\iota(\alpha) x = \theta(\alpha)^{-k'} x
\end{align}
for all $x \in A_{\pip'}$. By taking the $\pi$-formal completion of the schemes 
over $\Spec R$ in (\ref{alpha}), consider the following commutative diagram: 
\begin{align}
\label{nu}
\xymatrix{
\wh{X}_! \ar[d] \ar[r]^{\langle \alpha \rangle} &\wh{X}_! \ar[d] \\
\wh{\Jac} \ar[d]_\nu \ar[r]^{\langle \alpha \rangle} & \wh{\Jac} \ar[d]^\nu \\
\wh{A}_{\pip'} \ar[r]^{\iota(\alpha)} & \wh{A}_{\pip'}
}
\end{align}


\subsection{N\'{e}ron Models}
\label{neron}
For any scheme $X$ defined over $\Spec K$, let $\Ner{X}$ denote 
the N\'{e}ron model of $X$ over $\Spec R$. Then by the N\'{e}ron mapping 
property we have 
\begin{align}
\Ner{(\Mu)} \inj  \Ner{\Jac} \stk{\nu}{\longrightarrow} A_{\pip'}^{\ner}.
\end{align} 

Let ${(A_{\pip'}^{\ner})}^0$ denote the connected component to the identity of 
$A_{\pip'}^{\ner}$. Then ${(A_{\pip'}^{\ner})}^0$ satisfies the following short 
exact sequence of group schemes over $\Spec R$
\begin{align}
0 \longrightarrow T \stk{\iota}{\longrightarrow} {(A_{\pip'}^{\ner})}^{0} 
\stk{\sigma}{\longrightarrow} B \longrightarrow 0 
\end{align}
where $T$ is a torus and $B$ is an abelian scheme over $\Spec R$. Since we 
consider $R$ to be the maximal unramified extension of $\Ou_{\mfrak{p}}$, 
we can assume that $T$ is split over $R$ \cite{MR1083353}. 

By the N\'{e}ron mapping property, the Hecke and the diamond operators 
acting on $A_{\pip'}$ over $\Spec K$ induces endomorphisms on 
$A_{\pip'}^{\ner}$ over 
$\Spec R$. Then by the functoriality of the N\'{e}ron models for any 
$\gamma \in \End_R(A_{\pip'}^{\ner})$ induces an endomorphism 
(using the same notation) $\gamma: B \map B$ satisfying
$$\xymatrix{
{(A_{\pip'}^{\ner})}^0 \ar[d] \ar[r]^\gamma & {(A_{\pip'}^{\ner})}^0 \ar[d] \\
B \ar[r]^{\gamma} & B 
}$$

Then by composition we 
obtain the following map of the corresponding $\pi$-formal schemes
\begin{align}
\label{beta}
\beta: \ohat{X}_! \simeq  (\MU)^0 \longrightarrow (\ohat{\Ner{\Jac}})^0 
\stk{\nu}{\longrightarrow} (\ohat{A_{\pip'}^{\ner}})^0,
\end{align}
where $(\MU)^0$ denotes the connected component of the $\pi$-formal scheme
$\MU$ corresponding to $\ohat{X}_!$.
Then recall the following lemma obtained in Proposition 4.5 of \cite{MR2400054}.

\begin{lemma}
\label{theta}
(1) If $B \ne 0$ then there exists a non-zero $\Psi_2 \in 
\bX_2(J^2 (\ohat{B}))$
which is equivariant with respect to the Hecke and diamond operators.

In particular, we obtain $\Theta_2 : J^2(\ohat{A_{\pip'}^{\ner}})^0  \map \hG$ 
which is given by the following composition 
$$ 
J^2(\ohat{A_{\pip'}^{\ner}})^0 
\stk{J^2(\sigma)}{\longrightarrow} J^2(\ohat{B}) \stk{\Psi_2}{\longrightarrow}
 \hG.
$$

(2) If $B =0$ then we have 
 $A_{\pip'}^{\ner} = T \simeq \bb{G}_m^g$ for some $g$ since
$T$ is split. Therefore we obtain a non-zero $\Psi_1 \in \bX_1(J^1(\ohat{T}))$
which is equivariant with respect to the Hecke and diamond operators.

In particular, we obtain $\Theta_1: J^1(\ohat{A_{\pip'}^{\ner}})^0 
\map \hG$ which is given by the following 
$$
J^1(\ohat{A_{\pip'}^{\ner}})^0 = J^1(\ohat{T}) \stk{\Psi_1}{\longrightarrow} \hG.
$$

\end{lemma}

Given a non-zero differential character $\Theta_r \in \bX_r(A_{\pip'})$, 
applying the jet space functor to (\ref{nu}) and combining with lemma 
\ref{theta} we obtain the following:
\begin{align}
\label{Psichi}
\xymatrix{
J^r(\wh{X}_!) \ar[d] \ar[r]^{\langle \alpha \rangle} & J^r(\wh{X}_!) \ar[d] \\
J^r(\wh{\Jac}) \ar[d]_{J^r(\nu)} \ar[r]^{\langle \alpha \rangle} & 
J^r(\wh{\Jac}) \ar[d]^{J^r(\nu)} \\
J^r(\wh{A}_{\pip'}) \ar[d]_{\Theta_r} \ar[r]^{\iota(\alpha)} & 
J^r(\wh{A}_{\pip'})\ar[d]^{\Theta_r}
\\ 
\hG \ar[r]^{\chi(\iota(\alpha))} & \hG
}
\end{align}
where $\chi$ is the character from subring of $\End(A_{\pip'})$ generated by
the Hecke and diamond operators on 
$A_\pi$ to $R$ as constructed in \cite{MR2400054}, proposition 4.5.

\section{Main result}
\label{main-result}



We will first construct a presentation of the coordinate ring of the
$\pi$-formal scheme $\wh{X}_!$. This will be the analogous construction in 
\cite{MR2882615}. The construction will be done first modulo $\pi$ and then 
lift it in the formal scheme setting using Hensel's lemma.
But before we do that, we will prove the following basic lemma.

Let $X$ be a nonsingular curve over $k$.
Let $Y_1$ and $Y_2$ be  curves with $g_i: Y_i \rightarrow X$ for $i=1,2$ 
be smooth maps such that $\deg g_1 = \deg g_2$. Let $G$ be the 
Galois 
group acting on $Y_1$ such that $Y_1^G = X$. Let $f:Y_1 \rightarrow Y_2$ be a morphism
over $X$. Also further assume that $G$ acts transitively and freely 
on the fibers of $Y_2$ such that $f$ is $G$-equivariant. We have the 
following diagram
$$\xymatrix{
Y_1 \ar[dr]_{g_1} \ar[r]^f & Y_2 \ar[d]^{g_2} \\
& X.
}$$
The following lemma is standard:
\begin{lemma}
\label{gg}
Let $Y_1, Y_2, X$ and $f$ be as above. Then $f$ is an isomorphism.
\end{lemma}
\begin{proof}
Since $f$ is finite, it is sufficient to show that $f$ is a bijection at 
the level of $k$-points and that $f$ induces an injection on the tangent spaces.
Since both $Y_1$ and $Y_2$ are nonsingular curves over $k$ with 
$\deg (g_1) = \deg (g_2)$, it is sufficient to show that $f$ is bijective as
that would imply that the ramification index at each point has to be $1$.

Note that again by $\deg (g_1) = \deg (g_2)$, it is sufficient to show that
$f$ is injective. This we will show
over a fiber of any point $T \in X$. Let $P_0 \in Y_1$ such that 
$g_1(P_0)= T$.
Since $G$ acts transitively on the fibers of $Y_1$, any other $P$ lying over
$T$ will be given by $P= \sigma(P_0)$ for some $\sigma \in G$.

Now suppose $f$ is not injective over the fiber of $T$. Then there exist
distinct $P, P' \in Y_1$ such that $f(P)=f(P')$. Then if $P=\sigma (P_0)$ and
$P'=\sigma'(P_0)$, then we have $\sigma \sigma'^{-1} f(P_0) = f(P_0)$. But
since $G$ acts freely on the fibers of $Y_2$ as well, we must have 
$\sigma = \sigma'$ which is a contradiction and hence $f$ must be injective.
\end{proof}

If we still denote by $x$ the generator of $\epsilon^*\uomega$, then 
we have $a^+ = t x$ for some $t \in H^0(\Xp ,\Ou_{\Xp})$ and 
hence we obtain $t^{q-1} = \bar{\vp}$. 
Set $\overline{S}_{!!}= \barS[y]/(y^{q-1}- \obar{\vp})$
and $\Xpp:=\Spec \overline{S}_{!!}$. Define the $\overline{S}$-algebra
map $f^*:\overline{S}_{!!} \rightarrow \overline{S}_!$ given by $f^*(y)=t$.

Now we know that for any $\langle d \rangle \in G=\Z_p^{\times}$ the action
is given by
$\langle d \rangle a^+ = d^{-1} a^+$, which induces $\langle d\rangle t
= d^{-1} t$. We define an action of $G$ on $\Xpp $ by 
$\langle d \rangle y = d^{-1}y$. Then clearly $f^*$ is $G$-equivariant.
Let $f:\Xp \rightarrow \Xpp$ be the induced map of varieties.

\begin{prop}
\label{equal}
We have the following isomorphisms:
\begin{enumerate}
\item
 $\Xp \simeq \Xpp $.
 \item
In particular, we have
$\Ou(\Xfo) \simeq S[y]/(y^{q-1}-\vp)$ where $\vp \in S$ is some lift of 
$\obar{\vp}$.
 \end{enumerate}
\end{prop}

\begin{proof}
Note that $\Xpp$ is smooth over $\overline{X}$. Then $(1)$ 
follows from lemma \ref{gg} applied to $\Xp, \Xpp$ and $f$.
The second statement now follows from Lemma 3.2 of \cite{MR3349440}. 
\end{proof}

\begin{lemma}
\label{mainlemma}
If  we denote $S_n=\Ou(J^n\widehat{X})$, then we have
$$J^n(\wh{X}_!) = \Spf S_n[t]/(t^{q-1} - \varphi).$$
where $\varphi \in S$ is as in proposition \ref{equal}
\end{lemma}
\begin{proof}

Since $\widehat{X}_! \rightarrow \widehat{X}$ is \'etale, by proposition 1.6
in \cite{MR1748272} we have the 
identification $J^n(\wh{X}_!) \simeq J^n(\wh{X}) \times_X \wh{X}_!$
and the result follows from Proposition \ref{equal} (2). 
\end{proof}

Consider the linear map 
$$\tau:\Ou(J^n(\wh{X}_!) ) \longrightarrow \bigoplus_{r=0}^{q-2}  M^n(-r)$$
given by $t \rightarrow x^{-1}$ where $x$ is the generator of the invertible 
sheaf $\uomega$. 

\begin{prop}
\label{mainprop}
For each $n$, the map $\tau$ induces an isomorphism of $S_n$-modules: 
$$\tau:\Ou(J^n(\wh{X}_!) ) \simeq \bigoplus_{r=0}^{q-2}  M^n(-r)$$
\end{prop}
\begin{proof}
For any element $\alpha = \sum_{r=0}^{q-2} \alpha_r t^r \in 
\Ou(J^n(\wh{X}_!) )$ we have  $\tau(\alpha)= \sum \alpha_r x^{-r}$.
Then clearly this map is injective and surjective as well.
\end{proof}

\subsection{Proof of theorem \ref{main-thm-new}}
We recall $\omega_{\pip'}$ to be the weight $2$ Hecke cuspform
associated as the companion form of the mod $p$ modular form $\pip$.
Recall the morphism 
$\beta$ as in cf. \S\ref{beta}
$$ \beta: \ohat{X}_! \longrightarrow (\ohat{A_{\pip'}^{\ner}})^0.$$


Then for all $r$, this induces the associated map of $\pi$-formal jet spaces
\begin{align}
\label{beta2}
J^r(\beta) : J^r\ohat{X}_! \map J^r(\ohat{A_{\pip'}^{\ner}})^0.
\end{align}
Now when $B=0$ set $r=1$, and $r=2$ otherwise. Then 
we define $\fsharp \in \Ou(J^r\ohat{X}_!)$ as 
\begin{align}
\label{fsharp}
\fsharp := \Theta_r \circ J^r(\beta).
\end{align}
where $\Theta_r$ is as in lemma \ref{theta}.
Now recall the identification as in proposition \ref{mainprop}
\begin{align}
\label{oiso}
\Ou(J^r(\ohat{X})) \simeq \bigoplus_{j=0}^{q-2} M^r(-j)
\end{align}
where the isomorphism respects the action of the group of diamond operators
that is isomorphic to $(\Z/q\Z)^\times$.

Now we claim that $\fsharp$ belongs to one of the graded pieces in 
(\ref{oiso}). And we will do that by showing that $\fsharp \in 
\Ou(J^r(\ohat{X}_!))$ is an eigenform with respect to the group of diamond 
operators $(\Z/q\Z)^\times$. 
By (\ref{Psichi}) we have 
\begin{align}
\fsharp \circ \langle \alpha \rangle = \chi(\iota(\langle\alpha \rangle)) 
\fsharp.
\end{align}
Since $(\Z/q\Z)^\times$ is a cyclic group, there exists an integer
 $c$ coprime to $(q-1)$ 
such that $\chi(\iota(\alpha)) = (\theta(\langle \alpha \rangle))^{c(k-2)}$. 
Hence we have 
\begin{align}
\fsharp \circ \langle \alpha \rangle = 
(\theta(\langle \alpha \rangle))^{k'} \fsharp
\end{align}
where $k' \equiv c(k-2) \mod (q-1)$. Hence $\fsharp \in M^r(-k')$ and we 
are done.

\bibliographystyle{plain}
\bibliography{Eisensteinquestion1}

\end{document}